\documentclass[12pt]{amsart}

\usepackage{amsmath,amssymb,amsthm}
\usepackage{hyperref}
\usepackage{graphicx}%
\usepackage{multirow}%
\usepackage{amsmath,amssymb,amsfonts}%
\usepackage{amsthm}%
\usepackage{mathrsfs}%
\usepackage[title]{appendix}%
\usepackage{xcolor}%
\usepackage{textcomp}%
\usepackage{manyfoot}%
\usepackage{booktabs}%
\usepackage{algorithm}%
\usepackage{algorithmicx}%
\usepackage{algpseudocode}%
\usepackage{listings}%
\usepackage{mathtools}
\usepackage{tikz}
\usepackage{amscd} 
\usetikzlibrary{cd}
\numberwithin{equation}{section}

\theoremstyle{plain} 
\newtheorem{theorem}{Theorem}[section]
\newtheorem{proposition}[theorem]{Proposition}

\newtheorem{cor}[theorem]{Corollary}

\newtheoremstyle{romanstyle}
  {6pt}{6pt}   
  {\normalfont} 
  {}            
  {\normalfont} 
  {.}           
  {0.5em}       
  {}            
\theoremstyle{romanstyle}
\newtheorem{definition}[theorem]{Definition}

\newtheorem{example}[theorem]{Example}
\newtheorem{remark}[theorem]{Remark}
\def\conjconn#1{\overline{\nabla}^{\substack{\scalebox{0.4}{\phantom{i}}\\#1}}}
\def\partialt{\frac{\partial}{\partial t}}
\title[\scalebox{0.8}{The Second Variational Formula for Statistical Biharmonic Maps}]{The Second Variational Formula for Statistical Biharmonic Maps}

\author{Ryu Ueno}


\begin{document}

\begin{abstract}
Recently, the statistical bi-energy functional and its first variational formula were introduced by the author and H.~Furuhata.
The Maps satisfying the corresponding Euler--Lagrange equation are called statistical biharmonic maps.
We present the second variational formula for the statistical bi-energy functional and introduce the notion of stability.
If the target statistical manifold is a Hessian manifold, it turns out that the second variational formula can be represented using the Hessian curvature.
We provide examples of statistical biharmonic maps into Hessian manifolds that are significant in Hessian and information geometry.
We also determine the stability of these statistical biharmonic maps, including the improper affine sphere.
\end{abstract}

\maketitle

\section{Introduction}\label{sec1}
~~~ A \emph{statistical manifold} $(M,g,\nabla)$ is just a Riemannian manifold $(M,g)$ endowed with an affine connection $\nabla$ satisfying the Codazzi equation, and the pair $(g,\nabla)$ is called a \emph{statistical structure}.
If $\nabla$ is the Levi-Civita connection of $(M,g)$, then the statistical manifold is called \emph{Riemannian}, hence the simplest statistical manifolds are Riemannian manifolds.
Statistical structures were already known in the beginning of the 20th century from affine differential geometry \cite{NomizuSasaki}, but the geometry of statistical manifolds has been actively studied after 1982, when S. Amari and H. Nagaoka found information geometry by applying this structure to computer science \cite{MR1800071, AmariNagaoka1982}.
In information geometry, the statistical manifolds that appear are basically ones with a flat affine connection.
Such statistical manifolds are called \emph{Hessian manifolds}, and have been mainly studied by H. Shima, H. Furuhata, and T. Kurose \cite{FURUHATA2011S86, HF2013, shima2007geometry}.
Today, not only Hessian manifolds, but also various classes of statistical manifolds are the focus of research.\\

Recently in \cite{FU2025}, the first variational problem in the geometry of statistical manifolds was introduced.
Let $u:(M,g,\nabla^M)\to(N,h,\nabla^N)$ be a smooth map between statistical manifolds.
If $M$ is compact, we define the \emph{statistical bi-energy} by
\begin{equation*}
    E_2(u) = \int_{M} \|\tau(u)\|_h^2\,d\mu_g,
\end{equation*}
where $\tau(u)$ is the tension field of $u$.
The first variational formula of this functional induced \emph{statistical biharmonic maps}.
This formulation was inspired by the bi-energy functional and biharmonic maps in Riemannian geometry, for which the first and second variational formulas were obtained by G.Y. Jiang \cite{JB}.
One important example is the improper affine sphere $f:M^m\to\mathbb{R}^{m+1}$ in equiaffine geometry; it is a statistical biharmonic map $f:(M,g,\nabla)\to(\mathbb{R}^{m+1},g_0,D)$ with the induced statistical structure $(g,\nabla)$ and the Euclidean metric $g_0$ on $\mathbb{R}^{m+1}$.\\

We present the second variational formula of the statistical bi-energy functional and define the notion of stability.
Stability of biharmonic maps in Riemannian geometry is a widely studied topic \cite{VBAS, Sub, AMi, Yol}.
Moreover, since compact manifolds are limited in affine differential geometry and information geometry, we extend the notion of statistical biharmonic maps to non-compact statistical manifolds, as in Riemannian geometry \cite{EL,NHS2}.
Let $u \colon (M,g,\nabla^M) \to (N,h,\nabla^N)$ be a statistical biharmonic map, and $\Omega\subset M$ a relatively compact domain, that is, a connected open set with compact closure.
Denote by $E_{2,\Omega}(u)$ the statistical bi-energy measured on $\Omega$.
A smooth variation $F=\{u_t\}_{t\in(-\epsilon,\epsilon)}$ of $u$ is a smooth map $\mbox{$F:M\times(-\epsilon,\epsilon)\to N$}$ satisfying $u_0=u$, where we set $u_t(x) = F(x,t),\,x\in M$ for $\mbox{$t\in(-\epsilon,\epsilon)$}$.
The variational vector field $V\in\Gamma(u^{-1}TN)$ is defined by $V_x=\left.\frac{\partial}{\partial t}\right|_{t=0}u_t(x),$ $x\in M$.
We only consider variations $\{u_t\}$ of $u$ whose variational vector field has the support contained in $\Omega$.
For any $V\in\Gamma(u^{-1}TN)$ which has compact support in $\Omega$, there exists a smooth variation of $u$ such that the corresponding variational vector field is $V$.
By Theorem \ref{main}, the second derivative of $E_{2,\Omega}(u_t)$ is given by:
\begin{equation*}
    \begin{split}
        \left.\frac{d^2}{dt^2}\right|_{t=0} E_{2,\Omega}(u_t) =& \int_M \Big\|\Delta_u V + \sum_i R^N(V,u_*e_i)u_*e_i - 2K^N_{\tau(u)}V \Big\|_h^2 \, d\mu_g\\
                                            &+ \int_M \langle \mathcal{H}_u(V), V \rangle \, d\mu_g.
    \end{split}
\end{equation*}
Detailed explanations of the other notations will be given later in Section~\ref{s3}, however, $\mathcal{H}_u$ is an operator on $\Gamma(u^{-1}TN)$ determined by the map $u$.
If $(N,h,\nabla^N)$ is a Hessian manifold, $\mathcal{H}_u$ can be expressed in terms of the \textit{Hessian curvature}, which is the sectional curvature for Hessian manifolds.
In particular, if $u$ is a mapping to a Hessian manifold of Hessian curvature $0$, then $\mathcal{H}_u$ simply vanishes.
This gives applications of stability results of statistical biharmonic maps into Hessian manifolds.
Also, note that this second variational formula coincides with the one in \cite{JB}, if the target and source statistical manifolds are both Riemannian.\\

The paper is organized as follows. 
In \S 2, we collect the necessary background on statistical manifolds, including affine differential geometry and Hessian geometry. 
We review the statistical connection Laplacians, the statistical bi-energy, and the Euler--Lagrange equation defining the statistical biharmonic map in \S 3. 
We also give some examples of statistical biharmonic maps in this section.
The second variational formula will be proved in \S 4, and together we introduce the operator $\mathcal{H}_u$ for the map $u$ between statistical manifolds. 
Lastly, in \S 5 we give a condition under which the only weakly stable biharmonic maps are those with vanishing tension field.
We also compute the stability of statistical biharmonic maps given in \S 3, including the improper affine sphere.

\section{Preliminaries}\label{sec2}

\subsection{Statistical manifolds}

Throughout this paper, all objects are assumed to be smooth.
Let $M$ be an orientable manifold  of dimension $m \geq 2$.
Here, we denote by $C^{\infty}(M)$ the set of smooth functions on $M$, and by $\Gamma(E)$ the set of sections of a vector bundle $E$ over $M$.\\

\begin{definition}
    Let $g$ be a Riemannian metric and $\nabla$ a torsion-free affine connection on $M$. 
    The pair $(g,\nabla)$ is called a statistical structure on $M$ if it satisfies the Codazzi equation:
    \begin{equation*}
        (\nabla g)(Y,Z;X)=(\nabla_X g)(Y,Z)=(\nabla_Yg)(X,Z),\quad X, Y, Z \in \Gamma(TM).
    \end{equation*}
    The triplet $(M,g,\nabla)$ is called a \textit{statistical manifold}.\\
\end{definition}

On a Riemannian manifold $(M,g)$, we denote by $\nabla^g$ the Levi-Civita connection. 
Since $\nabla^gg=0$ holds, the triplet $(M,g,\nabla^g)$ is the simplest statistical manifold. 
In this case, the statistical manifold is called \textit{Riemannian}.\\

\newpage
\begin{definition}
Let $(M,g,\nabla)$ be a statistical manifold.

$(1)$ \ 
Define $K=K^{(g,\nabla)} \in \Gamma(TM^{(1,2)})$ by
$$
K(X,Y)= K_XY= \nabla_X Y-\nabla^g_X Y,
\quad X, Y \in \Gamma(TM).   
$$
We call $K$ the \textit{difference tensor} of $(M,g,\nabla)$.
We also denote it by $K^M$ if necessary. \\

$(2)$ \
The vector field $T=T^{(g, \nabla)}=\mathrm{tr}_gK \in \Gamma(TM)$ is called the {\sl Tchebychev vector field} of $(M,g,\nabla)$. 
We also denote it by $T^M$ if necessary.
Here, if $T=0$, then $(M, g, \nabla)$ is said to satisfy the {\sl equiaffine condition}\/ or the {\sl apolarity condition}\/.  \\

$(3)$ \ 
Define an affine connection $\overline{\nabla}$ by
$$
Xg(Y,Z)=g(\nabla_XY,Z)+g(Y, \overline{\nabla}_XZ),
\quad X, Y, Z \in \Gamma(TM),    
$$
and call it the {\sl conjugate connection}\/ of $\nabla$ with respect to $g$. 
Conjugate connections are also called {\sl dual connections}\/. 
It is easy to verify that $(M,g,\overline{\nabla})$ is also a statistical manifold.\\

$(4)$ \ 
Define the \textit{curvature tensor field} $R^\nabla \in \Gamma(TM^{(1,3)})$ 
of $\nabla$ by
$$
R^\nabla (X,Y)Z=\nabla_X \nabla_Y Z-\nabla_Y \nabla_X Z-\nabla_{[X,Y]}Z,
\quad X, Y, Z \in \Gamma(TM). 
$$
On a statistical manifold $(M,g,\nabla)$, we often denote $R^{\nabla}$ by $R$, 
$R^{\nabla^g}$ by $R^g$, and $R^{\overline{\nabla}}$ by $\overline{R}$, for short. \\

\end{definition}


\begin{definition}
    On a statistical manifold $(M,g,\nabla)$, define a tensor field $L=L^{(g,\nabla)} \in\Gamma(TM^{(1,3)})$ by
    $$
    g(L(Z,W)X,Y) = g(R(X,Y)Z,W),\quad X,Y,Z,W\in\Gamma(TM),
    $$
and call it the \textit{curvature interchange} tensor field of $(g,\nabla)$. 
We denote it by $L^M$ if necessary. 
We also denote $L^{(g, \overline{\nabla})}$ by $\bar{L}$.\\
\end{definition}

\begin{remark}
The following formulae hold for $X, Y, Z, W \in \Gamma(TM)$: 
\begin{eqnarray}
&&
\overline{\nabla}_X Y=\nabla^g_X Y-K(X,Y), \label{Kconjg}\\
&&
\nabla^g_X Y=\frac{\nabla_X Y+\overline{\nabla}_X Y}{2},\label{Leviconj}\\
&&
g(\overline{R}(X,Y)Z,W)=-g(Z, R(X,Y)W),\label{intcurvature}\\
&&
L(X,Y)Z = -\bar{L}(Y,X)Z,\label{swap}\\
&&
L(X,Z)Y - L(Y,Z)X = \overline{R}(X,Y)Z.\label{diffcurv}
\end{eqnarray}
\end{remark}

\vspace{0.5\baselineskip}

\begin{definition}
A statistical manifold $(M, g, \nabla)$ is said to be {\sl conjugate symmetric}\/
if the equality $R=L$ holds. \\
\end{definition}

\begin{remark}
The conjugate symmetry of $(M,g,\nabla)$ is equivalent to each of the following conditions:
\renewcommand{\labelenumi}{$(\theenumi)$}
\begin{enumerate}
    \item $R=\overline{R}$
    \item $\nabla^gK$ is symmetric on $M$.\\
\end{enumerate}
\end{remark}


\begin{definition}
    A statistical manifold $(M,g,\nabla)$ is said to be of \textit{constant sectional curvature} $\lambda\in\mathbb{R}$ if it satisfies
    \begin{equation*}
        R(X,Y)Z=\lambda(g(Y,Z)X-g(X,Z)Y),\quad X,Y,Z\in\Gamma(TM).
    \end{equation*}
\end{definition}

\vspace{0.5\baselineskip}

The following proposition is known for statistical manifolds of constant sectional curvature; see \cite{MR4452143}.\\
\begin{proposition}
    Let $(M,g,\nabla)$ be a statistical manifold. 
    The following conditions are mutually equivalent$:$
    \renewcommand{\labelenumi}{$\operatorname{(\theenumi)}$}
    \begin{enumerate}
        \item It has constant sectional curvature.
        \item It is conjugate symmetric and $\nabla$ is projectively flat.\\
    \end{enumerate}
\end{proposition}

    
Statistical manifolds naturally arise in affine differential geometry.
In particular, a statistical structure induced from a centroaffine immersion has constant sectional curvature $\pm1$, while a statistical structure induced from a Blaschke immersion satisfies the equiaffine condition.
Conversely, certain statistical structures can locally be induced from an affine immersion \cite{HM,NomizuSasaki}.

\vspace{0.5\baselineskip}

\begin{theorem}
    Let $(g,\nabla)$ be a statistical structure on a simply connected manifold $M$.
    Then, there exists a locally strongly convex equiaffine immersion $\{f,\xi\}$ that induces $(g,\nabla)$ on $M$ by the Gauss formula if and only if $\overline{\nabla}$ is projectively flat and if the Ricci tensor field of $\nabla$ is symmetric.\\
\end{theorem}

\subsection{Hessian manifolds}
Statistical manifolds of constant sectional curvature $0$ are called \textit{Hessian manifolds}.
Many statistical manifolds from information geometry are Hessian manifolds.\\

\begin{example}
    The Euclidean space $(\mathbb{R}^m,g_0)$ endowed with the Levi-Civita connection $\nabla^{g_0}$ is a Hessian manifold.\\
\end{example}

\begin{example}[\cite{HF2013}]
    \label{hess0}
    Let $\mathbb{R}^{+}=\{y\in\mathbb{R}\mid y>0\}$, and consider $(\mathbb{R}^{+})^m$ as a Riemannian submanifold of the Euclidean space $(\mathbb{R}^m,g_0)$.
    Define an affine connection $\nabla$ by 
    \begin{equation*}
        \nabla_{\frac{\partial}{\partial y^i}}\frac{\partial}{\partial y^j} = \frac{\delta_{ij}}{y^i}\frac{\partial}{\partial y^i}
    \end{equation*}
    where $\{y^1,\ldots,y^m\}$ is the standard coordinates on the Euclidean space restricted to $(\mathbb{R}^{+})^m$.
    Then, the triplet $((\mathbb{R}^{+})^m,g_0,\nabla)$ is a Hessian manifold.\\
\end{example}


\begin{example}[\cite{FIK}]
    \label{normdis}
    Consider the upper half plane $\mathbb{H}^2=\{(x,y)\in\mathbb{R}^2\mid y>0\}$.
    Define a Riemannian metric $g^F$ on $\mathbb{H}^2$ by
    \begin{equation*}
        g^F = \frac{dx^2+2dy^2}{y^2}
    \end{equation*}
    and a torsion-free affine connection $\nabla^{(m)}$ by
    \begin{equation*}
        \nabla^{(m)}_{\frac{\partial}{\partial x}}\frac{\partial}{\partial x} = \frac{1}{y}\frac{\partial}{\partial y},\quad \nabla^{(m)}_{\frac{\partial}{\partial x}}\frac{\partial}{\partial y} = 0,\quad \nabla^{(m)}_{\frac{\partial}{\partial y}}\frac{\partial}{\partial y} = \frac{1}{y}\frac{\partial}{\partial y}.
    \end{equation*}
    The triplet $(\mathbb{H}^2,g^F,\nabla^{(m)})$ is a Hessian manifold, called the \textit{statistical manifold of normal distributions}, where $g^F$ is called the \textit{Fisher metric}, and $\nabla^{(m)}$ the \textit{mixture connection}.
    The conjugate connection of $(\mathbb{H}^2,g^F,\nabla^{(m)})$ which we will denote by $\nabla^{(e)}$ is called the \textit{exponential connection}.
    The difference tensor $K^F$ of $(\mathbb{H}^2,g^F,\nabla^{(m)})$ is given by the following$\mathrm{:}$
    \begin{equation*}
        K^F_{\frac{\partial}{\partial x}}\frac{\partial}{\partial x} = \frac{1}{2y}\frac{\partial}{\partial y},\quad K^F_{\frac{\partial}{\partial x}}\frac{\partial}{\partial y}=\frac{1}{y}\frac{\partial}{\partial x},\quad K^F_{\frac{\partial}{\partial y}}\frac{\partial}{\partial y}=\frac{2}{y}\frac{\partial}{\partial y}.
    \end{equation*}
\end{example}

\vspace{0.5\baselineskip}

\begin{definition}
    Let $(M,g,\nabla)$ be a Hessian manifold.
    The tensor field $H\in\Gamma(TM^{(1,3)})$ defined by
    \begin{equation*}
        H(Y,Z;X) = (\nabla_X K)(Y,Z),\quad X,Y,Z\in\Gamma(TM)
    \end{equation*}
    is called the \textit{Hessian curvature tensor field}.
    Moreover, we say that $(M,g,\nabla)$ has \textit{constant Hessian curvature} $c$ if there exists a constant $c\in\mathbb{R}$ such that
    \begin{equation}
        \label{chcc}
        H(Y,Z;X) = -\frac{c}{2}(g(X,Y)Z+g(X,Z)Y)
    \end{equation}
    for any $X,Y,Z\in\Gamma(TM)$.
    We abbreviate this condition as CHC $c$.\\
\end{definition}

\begin{proposition}[\cite{shima2007geometry}]
    A Hessian manifold of CHC $c$ has Riemannian metric of constant curvature $-\frac{c}{4}$.
\end{proposition}
\begin{proof}
    On any statistical manifold $(M,g,\nabla)$, it holds that
    \begin{align}
    R(X,Y)Z &= R^g(X,Y)Z + (\nabla^g_X K)(Y,Z) - (\nabla^g_Y K)(X,Z) + [K_X,K_Y]Z, \label{curveq1}\\
    R(X,Y)Z &= R^g(X,Y)Z + (\nabla_X K)(Y,Z) - (\nabla_Y K)(X,Z) - [K_X,K_Y]Z. \label{curveeq2}
    \end{align}
    for any $X,Y,Z\in\Gamma(TM)$.
    Therefore, if $(M,g,\nabla)$ is Hessian, the following equation holds for $R^g$$\mathrm{:}$
    \begin{equation*}
        R^g(X,Y)Z = -\frac{1}{2}\left((\nabla_X K)(Y,Z) - (\nabla_Y K)(X,Z)\right),\quad X,Y,Z\in\Gamma(TM).
    \end{equation*}
    By substituting \eqref{chcc}, we obtain the desired result.\\
\end{proof}

\begin{example}
    The Euclidean space $(\mathbb{R}^m,g_0)$, equipped with its Levi-Civita connection is a Hessian manifold of CHC $0$.\\
\end{example}

\begin{example}
    Let $\overline{\nabla}$ be the conjugate connection of the Hessian manifold $((\mathbb{R}^{+})^m,g_0,\nabla)$ in Example~$\ref{hess0}$.
    The statistical manifold $((\mathbb{R}^{+})^m,g_0,\overline{\nabla})$ is a Hessian manifold of CHC $0$.
    In fact, by \cite{HF2013}, the product statistical manifold of $((\mathbb{R}^{+})^m,g_0,\overline{\nabla})$ and the Euclidean space is regarded as the maximal Hessian manifold of CHC $0$.\\
\end{example}

\begin{example}
    The statistical manifold of normal distributions is a Hessian manifold of CHC $2$.\\
\end{example}


\section{Statistical biharmonic maps}
\label{s3}
~~~We introduce the notion of statistical biharmonic maps which arise as the Euler-Lagrange equation of the statistical bi-energy functional and provide their definition.\\
\subsection{Connections on the vector bundles over statistical manifolds}
~~~We summarize the notations and definitions introduced in \cite{FU2025} in this subsection.
Let $u:(M,g,\nabla^M)\to(N,h,\nabla^N)$ be a smooth mapping between statistical manifolds. 
On the vector bundle $u^{-1}TN$, we denote by $\nabla^u$ the connection induced from $\nabla^N$, that is, $\nabla^u_X U=\nabla^N_{u_* X}U \in \Gamma(u^{-1}TN)$ for 
$X \in \Gamma(TM)$ and $U \in \Gamma(u^{-1}TN)$.
We also denote by $\overline{\nabla}^{\substack{\scalebox{0.45}{\phantom{i}}\\u}}$ and $\widehat{\nabla}^u$ the connections on $u^{-1}TN$ induced by $\overline{\nabla}^{\substack{\scalebox{0.3}{\phantom{i}}\\N}}$, $\nabla^h$, respectively.\\

\begin{definition}[\cite{FU2025}]
\label{FUdef}
Let $(M, g, \nabla^M)$ be a statistical manifold, and $E \to M$ a vector bundle with connection $\nabla^E : \Gamma(E) \to \Gamma(E \otimes T^*M)$. 
We define the operator $\Delta^E=\Delta^{(g, \nabla^M, \nabla^E)} : \Gamma(E) \to \Gamma(E)$ as 
\begin{equation*}
    \Delta^E \xi 
    =\mathrm{tr}_g \left\{ (X,Y) \mapsto \nabla^E_X\nabla^E_Y \xi -\nabla^E_{\nabla^M_X Y}\xi \right\}, 
\end{equation*}
and call it the {\sl statistical connection Laplacian}\/ with respect to $(g, \nabla^M)$ and $\nabla^E$. 
If there is a fiber metric $\langle\cdot,\cdot\rangle$ on $E$, then the statistical connection Laplacian with respect to $(g,\overline{\nabla}^{\substack{\scalebox{0.3}{\phantom{i}}\\M}})$ and $\overline{\nabla}^{\substack{\scalebox{0.3}{\phantom{i}}\\E}}$ is denoted as $\bar\Delta^E$, where $\overline{\nabla}^{\substack{\scalebox{0.3}{\phantom{i}}\\E}}$ is defined by the following equation:
\begin{equation*}
    X\langle \xi,\eta\rangle = \langle\nabla^E_X\xi,\eta\rangle + \langle\xi,\overline{\nabla}^{\substack{\scalebox{0.3}{\phantom{i}}\\E}}_X\eta\rangle,\quad \xi,\eta\in\Gamma(E).
\end{equation*}
\end{definition}

\vspace{0.5\baselineskip}

If $u:(M,g,\nabla^M)\to(N,h,\nabla^N)$ is a smooth mapping between statistical manifolds, we denote by $\Delta^u$ the statistical connection Laplacian with respect to $(g,\nabla^M)$ and $\nabla^u$.
We also denote by $\bar{\Delta}^u$ the statistical connection Laplacian associated with $(g,\overline{\nabla}^{\substack{\scalebox{0.3}{\phantom{i}}\\M}})$ and $\overline{\nabla}^{\substack{\scalebox{0.45}{\phantom{i}}\\u}}$, and by $\widehat{\Delta}^u$ that associated with $(g,\nabla^g)$ and $\widehat{\nabla}^u$.\\


For a vector bundle $E$ over $M$, the support of $\xi\in\Gamma(E)$ is the closure of $\{p\in M\,|\,\xi_p\neq0\}$.
For a domain $\Omega\subset M$, denote by $\Gamma_{\Omega}(E)$ the smooth sections of $E$ with compact support in $\Omega$.
The next proposition can be proved in the same manner as of Proposition~2.14 in \cite{FU2025}.\\


\begin{proposition}
    \label{laplaceex}
    Let $E$ be a vector bundle over a statistical manifold $(M,g,\nabla^M)$ equipped with a connection $\nabla^E$ and a fiber metric $\langle\cdot,\cdot\rangle$, and $\Omega\subset M$ a relatively compact domain.
    For any $\xi,\eta\in\Gamma(E)$ which is either $\xi\in\Gamma_{\Omega}(E)$ or $\eta\in\Gamma_{\Omega}(E)$, the following formula holds$:$
    \begin{equation*}
        \int_{\Omega} \langle\Delta^E\xi,\eta\rangle d\mu_g = \int_{\Omega} \langle\xi,\bar{\Delta}^E\eta\rangle d\mu_g + \int_{\Omega} \operatorname{div}^g(T^M)\langle\xi,\eta\rangle d\mu_g.
    \end{equation*}
\end{proposition}

\vspace{1.5pt}

\subsection{Statistical biharmonic maps}
\label{sbm}
~~~A class of mappings between statistical manifolds; statistical biharmonic maps were introduced in \cite{FU2025}.
Let $(M,g,\nabla^M)$ and $(N,h,\nabla^N)$ be statistical manifolds.
For a smooth mapping $u:M\to N$, the tension field $\tau(u)=\tau^{(g,\nabla^M,\nabla^N)}$ of $u$ is defined by 
\begin{equation*}
\tau(u) = \mathrm{tr}_g\{(X,Y)\mapsto\nabla^u_X u_*Y-u_*\nabla^M_X Y\} 
        \in \Gamma(u^{-1}TN).
\end{equation*}
Here, notice that the definition of $\tau(u)$ is independent of $h$, therefore the equation $\tau(u)=0$ is only a characterization of mappings between $(M,g,\nabla^M)$ and $(N,\nabla^N)$.\\


\begin{example}
    Let $\nabla$ be an affine connection on $M$.
    The tension field of a curve $\gamma:(I,g_0,\nabla^{g_0})\to (M,\nabla)$ is $\tau(\gamma)=\nabla_{\dot{\gamma}}\dot{\gamma}$, where $g_0$ is the Euclidean metric restricted to $I\subset\mathbb{R}$.\\
\end{example}

\begin{example}[\cite{FU2025}]
    \label{affimm}
    Let $\{f,\xi\}$ be a locally strongly convex equiaffine immersion, where $f:M\to\mathbb{R}^{m+1}$.
    Denote by $D$ the standard flat affine connection of $\mathbb{R}^{m+1}$.
    A statistical structure $(g,\nabla)$ on $M$ is induced by the Gauss formula$:$
    \begin{equation*}
        D_Xf_*Y=f_*\nabla_XY+g(X,Y)\xi,\quad X,Y\in\Gamma(TM).
    \end{equation*}
    The tension field of $f:(M,g,\nabla)\to(\mathbb{R}^{m+1},D)$ is $\tau(f)=m\cdot\xi$.\\
\end{example}


We define a functional for mappings between statistical manifolds using the tension field.\\

\begin{definition}
    Let $u:(M,g,\nabla^M)\to (N,h,\nabla^N)$ be a smooth mapping and $\Omega\subset M$ a relatively compact domain.
    The statistical bi-energy $E_{2,\Omega}(u)=E_{2,\Omega}^{(g,\nabla^M,h,\nabla^N)}(u)$ of $u$ is defined by
    \begin{equation*}
        E_{2,\Omega}(u) = \int_{\Omega} \|\tau(u)\|_h^2\,d\mu_g,
    \end{equation*}
    where $d\mu_g$ is the measure derived from the Riemannian metric $g$.
    The statistical bi-tension field $\tau_2(u)=\tau_2^{(g,\nabla^M, h,\nabla^N)}(u)$ of $u$ is defined by
\begin{eqnarray*}
\tau_2(u)
&=&        
\bar{\Delta}^u\tau(u)+\operatorname{div}^g
(T^M)\tau(u) \nonumber\\
&&\quad 
-\sum_{i=1}^mL^N(u_*e_i,\tau(u)) u_*e_i -K^N(\tau(u),\tau(u) )
\quad 
\in \Gamma(u^{-1}TN),      
\end{eqnarray*}
where $\{e_1,\ldots,e_m\}$ is an orthonormal frame with respect to $g$.\\
\end{definition}

The following theorem can be proved in the same manner as Theorem 3.1 in \cite{FU2025}.\\

\begin{theorem}
    \label{thm1}
Let $u:(M,g,\nabla^M)\to(N,h,\nabla^N)$ be a mapping between statistical manifolds.
For any relatively compact domain $\Omega\subset M$ and a variation $\{u_t\}$ with $V=\left.\frac{\partial}{\partial t}\right|_{t=0}u_t\in\Gamma_{\Omega}(u^{-1}TN)$,
the first variational formula of the statistical bi-energy is given by
\begin{equation*}
  \left.\displaystyle\frac{d}{dt}\right|_{t=0}E_{2,\Omega}(u_t)
    =\int_{\Omega}\left\langle V,\tau_2(u)\right\rangle d\mu_g.
\end{equation*}
\end{theorem}


\vspace{1em}

\begin{definition}
    A mapping $u:(M,g,\nabla^M)\to(N,h,\nabla^N)$ that satisfies $\tau_2(u)=0$ is called a \textit{statistical biharmonic map}.\\
\end{definition}

We now present some examples of statistical biharmonic maps.\\

\begin{example}
    Every mapping $u:(M,g,\nabla^M)\to(N,h,\nabla^N)$ that satisfies $\tau(u)=0$ is a statistical biharmonic map.\\
\end{example}

\begin{example}[\cite{FU2025}]
    \label{impropaff}
    Consider the locally strongly convex equiaffine immersion $\{f,\xi\}$ in Example~\ref{affimm}.
    Equip the space $(\mathbb{R}^{m+1},D)$ with the Euclidean metric $g_0$, that is, a Riemannian metric which satisfies $Dg_0=0$.
    Then, the statistical bi-tension field of the map $u:(M^m,g,\nabla)\to(\mathbb{R}^{m+1},g_0,D)$ is given by
    \begin{equation*}
        \frac{1}{m}\tau_2(u) = -f_*(\mathrm{tr}_g\nabla S) - \mathrm{tr}S\xi - \operatorname{div}^g(T^M)\xi.
    \end{equation*}
    Where $S$ is the shape operator of $\{f,\xi\}$ defined by the Weingarten formula$:$
    \begin{equation*}
        D_X\xi = -f_*(SX),\quad X\in\Gamma(TM).
    \end{equation*}
    As a typical example, improper affine spheres satisfy $T^M=0$ and $S=0$, and hence they are statistical biharmonic maps.\\
\end{example}

\begin{example}[\cite{FU2025}]
    \label{curveex}
    Let $\gamma:((a,b),g_0,\nabla^{g_0})\to(M,g,\nabla)$ be a mapping between statistical manifolds, where $g_0$ is the Euclidean metric restricted to $(a,b)$.
    Denote by $\dot{\gamma}$ the velocity of $\gamma$.
    By Example~$\mathrm{4.1}$ of \textnormal{\cite{FU2025}}, the statistical bi-tension field $\tau_2(\gamma)$ of $\gamma$ is given by
    \begin{equation}
        \label{curvebiten}
        \tau_2(\gamma) = \overline{\nabla}_{\dot{\gamma}}\overline{\nabla}_{\dot{\gamma}}\nabla_{\dot{\gamma}}\dot{\gamma} - L(\dot{\gamma},\nabla_{\dot{\gamma}}\dot{\gamma})\dot{\gamma} - K(\nabla_{\dot{\gamma}}\dot{\gamma},\nabla_{\dot{\gamma}}\dot{\gamma}).
    \end{equation}
\end{example}

\vspace{1.5pt}

\begin{example}
    \label{curvhess0}
    Consider the Hessian manifold $((\mathbb{R}^{+})^m,g_0,\nabla)$ in Example~$\mathrm{\ref{hess0}}$.
    The curve $\gamma(t) = (\lambda_1t^2,\ldots,\lambda_mt^2),\,t\in(0,\infty)$ on $((\mathbb{R}^{+})^m,g_0,\nabla)$ satisfies equation $(\ref{curvebiten})$ for any positive constants $\lambda_1,\ldots,\lambda_m$.\\
\end{example}

We introduce new types of curves on statistical manifolds that are statistical biharmonic maps.\\

\begin{proposition}
    Let $\gamma(t),\,t\in(a,b)$ be a geodesic of $\nabla^g$ on a conjugate symmetric statistical manifold $(M,g,\nabla^M)$.
    If there exists a constant $\lambda\in\mathbb{R}$ such that $\nabla_{\dot{\gamma}}\dot{\gamma}=\lambda\cdot\dot{\gamma}$, then the map $\gamma:((a,b),g_0,\nabla^{g_0})\to(M,g,\nabla)$ is a statistical biharmonic map.\\
\end{proposition}

\begin{proof}
    Since $\nabla^g_{\dot{\gamma}}\dot{\gamma}=0$ holds, we have that
    \begin{equation*}
        \overline{\nabla}_{\dot{\gamma}}\dot{\gamma} = -\lambda\cdot\dot{\gamma},\quad K_{\dot{\gamma}}{\dot{\gamma}} = \lambda\cdot\dot{\gamma}.
    \end{equation*}
    Therefore, by equation $(\ref{curvebiten})$, the statistical bi-tension field $\tau_2(\gamma)$ of $\gamma$ is
    \begin{equation*}
            \tau_2(\gamma) = \lambda\cdot\overline{\nabla}_{\dot{\gamma}}\overline{\nabla}_{\dot{\gamma}}\dot{\gamma} - \lambda\cdot R(\dot{\gamma},\dot{\gamma})\dot{\gamma} - \lambda^2K_{\dot{\gamma}}\dot{\gamma}=0.
    \end{equation*}
\end{proof}

\begin{example}
    \label{projcurv}
    Consider the statistical manifold $(\mathbb{H}^2,g^F,\nabla^{(m)})$ of normal distributions in Example~\ref{normdis}.
    The curve $\gamma(t)=(0,e^{\lambda t}),\,t\in\mathbb{R}$ satisfies $\nabla^{(m)}_{\dot{\gamma}}\dot{\gamma}=2\lambda\dot{\gamma}$ for any constant number $\lambda\neq0$.
    It also satisfies $\nabla^{(e)}_{\dot{\gamma}}\dot{\gamma}=-2\lambda\dot{\gamma}$.
\end{example}

\section{The second variation formula of the statistical bi-energy}
~~~We now compute the second derivative of the statistical bi-energy.\\
\begin{definition}
    For a map $u:(M,g,\nabla^M)\to (N,h,\nabla^N)$ between statistical manifolds, define the operator $\mathcal{H}_u=\mathcal{H}^{(g,\nabla^M,h,\nabla^N)}_u$ by the following equation$:$
    \begin{equation*}
        \begin{split}
            \mathcal{H}_u(V) &=\overline{R}^{\substack{\scalebox{0.4}{\phantom{i}}\\N}}\left(V,\overline{\tau}(u)\right)\tau(u) + \sum_{i=1}^m\left(\overline{\nabla}^{\substack{\scalebox{0.4}{\phantom{i}}\\N}}_{u_*{e_i}}\overline{R}^{\substack{\scalebox{0.4}{\phantom{i}}\\N}}\right)\left(V,u_*{e_i}\right)\tau(u)\\
            &+2\sum_{i=1}^m\overline{R}^{\substack{\scalebox{0.4}{\phantom{i}}\\N}}\left(V,u_*{e_i}\right)\overline{\nabla}^{\substack{\scalebox{0.4}{\phantom{i}}\\u}}_{e_i}\tau(u)+\sum_{i=1}^m\left(\conjconn{N}_{V}\bar{L}^N\right)(\tau(u),u_*e_i)u_*e_i\\
            &-2\sum_{i=1}^mL^N(u_*e_i,\tau(u))\conjconn{u}_{e_i}V - \left(\conjconn{N}_{V}K^N\right)\left(\tau(u),\tau(u)\right),
        \end{split}
    \end{equation*}
    where $V\in\Gamma(u^{-1}TN)$.
    Here, $\overline{\tau}(u)\in\Gamma(u^{-1}TN)$ is the tension field of the map $u:(M,g,\overline{\nabla}^{\substack{\scalebox{0.4}{\phantom{i}}\\M}})\to(N,\overline{\nabla}^{\substack{\scalebox{0.4}{\phantom{i}}\\N}})$.\\
\end{definition}

\begin{proposition}
    \label{Joper}
    Let $u:(M,g,\nabla^M)\to (N,h,\nabla^N)$ be a map between statistical manifolds.
    If $(N,h,\nabla^N)$ is conjugate symmetric, then for any $V\in\Gamma(u^{-1}TN)$, the inner product $\langle\mathcal{H}_u(V),V\rangle$ reduces to the following$\mathrm{:}$
    \begin{equation*}
\begin{split}
\langle\mathcal{H}_u(V),V\rangle
= \Big\langle &\,R^N(V,\overline{\tau}(u))\tau(u)
+2\sum_{i=1}^m R^N(V,u_*e_i)\widehat{\nabla}^u_{e_i}\tau(u) \\
&+\sum_{i=1}^m \left(\nabla^h_{\tau(u)}R^N\right)(V,u_*e_i)u_*e_i - 2\sum_{i=1}^mR^N(u_*e_i,\tau(u))\widehat{\nabla}^u_{e_i}V \\
&-\left(\conjconn{N}_{V}K^N\right)(\tau(u),\tau(u))\,,\, V \Big\rangle.
\end{split}
\end{equation*}
\end{proposition}

\begin{proof}
    Let $V\in \Gamma(u^{-1}TN)$.
    Since we have $R^N=\overline{R}^{\substack{\scalebox{0.4}{\phantom{i}}\\N}}$, it holds that
    \begin{equation*}
        \begin{split}
            \bigg\langle\left(\overline{\nabla}^{\substack{\scalebox{0.4}{\phantom{i}}\\N}}_{u_*{e_i}}\overline{R}^{\substack{\scalebox{0.4}{\phantom{i}}\\N}}\right)&\left(V,u_*{e_i}\right)\tau(u),V\bigg\rangle = \bigg\langle \left(\conjconn{N}_{u_*e_i}R^N\right)\left(V,u_*{e_i}\right)\tau(u),V\bigg\rangle\\
            &=-\bigg\langle\left(\overline{\nabla}^{\substack{\scalebox{0.4}{\phantom{i}}\\N}}_{u_*{e_i}}R^N\right)\left(\tau(u),V\right)u_*{e_i}+2R^N(u_*{e_i},\tau(u))K^N_{u_*{e_i}}V,V\bigg\rangle.
        \end{split}
    \end{equation*}
    By substituting $R^N=L^N=\overline{R}^{\substack{\scalebox{0.4}{\phantom{i}}\\N}}=\bar{L}^N$ and using the second Bianchi identity of $\overline{\nabla}^{\substack{\scalebox{0.4}{\phantom{i}}\\N}}\overline{R}^{\substack{\scalebox{0.4}{\phantom{i}}\\N}}$, the inner product reduces to 
    \begin{equation}
        \label{redu1}
        \begin{split}
            \langle\mathcal{H}_u(V),V\rangle &= \bigg\langle R^N(V,\overline{\tau}(u))\tau(u) + 2\sum_{i=1}^mR^N(V,u_*e_i)\conjconn{u}_{e_i}\tau(u)\\
            &+\sum_{i=1}^m\left(\conjconn{N}_{\tau(u)}R^N\right)(V,u_*e_i)u_*e_i - 2\sum_{i=1}^mR^N(u_*e_i,\tau(u))\widehat{\nabla}^u_{e_i}V\\
            &- \left(\overline{\nabla}^{\substack{\scalebox{0.4}{\phantom{i}}\\N}}_{V}K^N\right)(\tau(u),\tau(u)),V\bigg\rangle.\\
        \end{split}
    \end{equation}
    Lastly, since $\conjconn{N} = \nabla^h - K^N$, we obtain
    \begin{equation}
        \label{redu2}
        \begin{split}
            \bigg\langle\left(\conjconn{N}_{\tau(u)}R^N\right)&(V,u_*e_i)u_*e_i,V\bigg\rangle\\
            &=\bigg\langle\left(\nabla^h_{\tau(u)}R^N\right)(V,u_*e_i)u_*e_i+2R^N(V,u_*{e_i})K^N_{u_*{e_i}}\tau(u),V\bigg\rangle
        \end{split}
    \end{equation}
    The proof is completed by substituting equation $(\ref{redu2})$ into equation $(\ref{redu1})$.\\
\end{proof}

\begin{theorem}
    \label{main}
    Let $u:(M,g,\nabla^M)\to (N,h,\nabla^N)$ be a statistical biharmonic map, and $\Omega\subset M$ a relatively compact domain.
    A smooth variation $F:M\times(-\epsilon,\epsilon)\to N$ of $u$ is a smooth map that satisfies $u_0=u$, where $u_t(x)=F(x,t),\,(x,t)\in M\times(-\epsilon,\epsilon)$.
    For each smooth variation $F=\{u_t\}$ of $u$, which the generated variational vector field $V=\left.\frac{\partial}{\partial t}\right|_{t=0}u_t$ has compact support in $\Omega$, the following equation holds$\mathrm{:}$
    \begin{equation}
        \label{secvar}
        \begin{split}
        \left.\frac{d^2}{dt^2}\right|_{t=0} E_{2,\Omega}(u_t) =& \int_{\Omega} \biggl\|\Delta^u V + \sum^m_{i=1}R^N(V,u_*{e_i})u_*{e_i}-2K^N_{\tau(u)}V\biggr\|^2_h\,d\mu_g\\
        &+\int_{\Omega} \langle \mathcal{H}_u(V),V \rangle\,d\mu_g.
        \end{split}
    \end{equation}
\end{theorem}

\vspace{0.5em}

\begin{proof}
    We compute the variation following \cite{FU2025}.
    By equipping $(-\epsilon, \epsilon) \subset \mathbb{R}$ with the standard Euclidean metric and its Levi-Civita connection, we define the affine connection $\nabla^B$ on the product statistical manifold $B=M\times (-\epsilon,\epsilon)$ by the following$\mathrm{:}$
    \begin{equation}
        \label{connB}
        \nabla^B_X Y = \nabla^M_X Y,\quad \nabla^B_{\frac{\partial}{\partial t}} X = \nabla^B_X \frac{\partial}{\partial t} = \nabla^B_{\frac{\partial}{\partial t}} \frac{\partial}{\partial t} =0,
    \end{equation}
    where we distinguish $X,Y\in\Gamma(TM)$ by $(X,0),(Y,0)\in\Gamma(T(M\times (-\epsilon,\epsilon)))$, respectively.
    Thus, $T^B=T^M$ holds.
    From Theorem \ref{thm1}, we have
    \begin{equation}
        \begin{split}
            \left.\frac{d^2}{dt^2}\right|_{t=0}& E_{2,\Omega}(u_t) = \left.\frac{d}{dt}\right|_{t=0}\int_{\Omega}\left\langle F_*\frac{\partial}{\partial t},\tau_2(u_t)\right\rangle\,d\mu_g\\
            &=\int_{\Omega} \left.\left\langle\nabla^F_{\frac{\partial}{\partial t}} F_*\frac{\partial}{\partial t},\tau_2(u_t)\right\rangle\right|_{t=0}\,d\mu_g + \int_{\Omega}\left.\left\langle F_*\frac{\partial}{\partial t},\overline{\nabla}^{\substack{\scalebox{0.4}{\phantom{i}}\\F}}_{\frac{\partial}{\partial t}}\tau_2(u_t)\right\rangle\right|_{t=0}\,d\mu_g.
        \end{split}
    \end{equation}
    Here, the first item in the last row vanishes since we have $\tau_2(u)=0$.
    We proceed by computing $\displaystyle \overline{\nabla}^{\substack{\scalebox{0.4}{\phantom{i}}\\F}}_{\frac{\partial}{\partial t}}\tau_2(u_t)$ for each $t\in(-\epsilon,\epsilon)$.
    First, we have
    \begin{equation}
        \label{fireq}
        \begin{split}
            \overline{\nabla}^{\substack{\scalebox{0.4}{\phantom{i}}\\F}}_{\frac{\partial}{\partial t}}\bar{\Delta}^{u_t}\tau(u_t) =& \bar{\Delta}^{u_t}\overline{\nabla}^{\substack{\scalebox{0.4}{\phantom{i}}\\F}}_{\frac{\partial}{\partial t}}\tau(u_t) + \sum_{i=1}^m\left(\overline{\nabla}^{\substack{\scalebox{0.4}{\phantom{i}}\\N}}_{F_*{e_i}}\overline{R}^{\substack{\scalebox{0.4}{\phantom{i}}\\N}}\right)\left(F_*\frac{\partial}{\partial t},F_*{e_i}\right)\tau(u_t)\\
            &+\overline{R}^{\substack{\scalebox{0.4}{\phantom{i}}\\N}}\left(F_*\frac{\partial}{\partial t},\overline{\tau}(u_t)\right)\tau(u_t) + \sum_{i=1}^m\overline{R}^{\substack{\scalebox{0.4}{\phantom{i}}\\N}}\left(\conjconn{F}_{e_i}F_*\partialt,F_*e_i\right)\tau(u_t)\\
            &+2\sum_{i=1}^m\overline{R}^{\substack{\scalebox{0.4}{\phantom{i}}\\N}}\left(F_*\frac{\partial}{\partial t},F_*{e_i}\right)\overline{\nabla}^{\substack{\scalebox{0.4}{\phantom{i}}\\u_t}}_{e_i}\tau(u_t).
        \end{split}
    \end{equation}
    This equation holds since we have
    \begin{equation*}
        \begin{split}
            \overline{\nabla}^{\substack{\scalebox{0.4}{\phantom{i}}\\F}}_{\frac{\partial}{\partial t}}&\overline{\nabla}^{\substack{\scalebox{0.4}{\phantom{i}}\\F}}_{e_i}\overline{\nabla}^{\substack{\scalebox{0.4}{\phantom{i}}\\F}}_{e_i}\tau(u_t)=\overline{\nabla}^{\substack{\scalebox{0.4}{\phantom{i}}\\F}}_{e_i}\overline{\nabla}^{\substack{\scalebox{0.4}{\phantom{i}}\\F}}_{\frac{\partial}{\partial t}}\overline{\nabla}^{\substack{\scalebox{0.4}{\phantom{i}}\\F}}_{e_i}\tau(u_t)+\overline{R}^{\substack{\scalebox{0.4}{\phantom{i}}\\N}}\left(F_*\frac{\partial}{\partial t},F_*e_i\right)\overline{\nabla}^{\substack{\scalebox{0.4}{\phantom{i}}\\F}}_{e_i}\tau(u_t)\\
            &=\overline{\nabla}^{\substack{\scalebox{0.4}{\phantom{i}}\\F}}_{e_i}\overline{\nabla}^{\substack{\scalebox{0.4}{\phantom{i}}\\F}}_{e_i}\overline{\nabla}^{\substack{\scalebox{0.4}{\phantom{i}}\\F}}_{\frac{\partial}{\partial t}}\tau(u_t) + \overline{\nabla}^{\substack{\scalebox{0.4}{\phantom{i}}\\F}}_{e_i}\overline{R}^{\substack{\scalebox{0.4}{\phantom{i}}\\N}}\left(F_*\frac{\partial}{\partial t},F_*e_i\right)\tau(u_t)\\
            &\quad +\overline{R}^{\substack{\scalebox{0.4}{\phantom{i}}\\N}}\left(F_*\frac{\partial}{\partial t},F_*e_i\right)\overline{\nabla}^{\substack{\scalebox{0.4}{\phantom{i}}\\F}}_{e_i}\tau(u_t),
        \end{split}
    \end{equation*}
    \begin{equation*}
            \conjconn{F}_{\partialt}\conjconn{F}_{\overline{\nabla}^{\substack{\scalebox{0.1}{\phantom{i}}\\ \scalebox{0.55}{$M$}}}_{e_i}e_i}\tau(u_t) = \conjconn{F}_{\overline{\nabla}^{\substack{\scalebox{0.1}{\phantom{i}}\\ \scalebox{0.55}{$M$}}}_{e_i}e_i}\conjconn{F}_{\partialt}\tau(u_t) + \overline{R}^{\substack{\scalebox{0.4}{\phantom{i}}\\N}}\left(F_*\partialt,u_{t*}\conjconn{M}_{e_i}e_i\right)\tau(u_t),
    \end{equation*}
    and 
    \begin{equation*}
        \begin{split}
            \overline{\nabla}^{\substack{\scalebox{0.4}{\phantom{i}}\\F}}_{e_i}\overline{R}^{\substack{\scalebox{0.4}{\phantom{i}}\\N}}&\left(F_*\frac{\partial}{\partial t},F_*e_i\right)\tau(u_t) \\
            =&\left(\overline{\nabla}^{\substack{\scalebox{0.4}{\phantom{i}}\\N}}_{F_*{e_i}}\overline{R}^{\substack{\scalebox{0.4}{\phantom{i}}\\N}}\right)\left(F_*\frac{\partial}{\partial t},F_*{e_i}\right)\tau(u_t)+\overline{R}^{\substack{\scalebox{0.4}{\phantom{i}}\\N}}\left(\conjconn{F}_{e_i}F_*\partialt,F_*e_i\right)\tau(u_t)\\
            &+\overline{R}^{\substack{\scalebox{0.4}{\phantom{i}}\\N}}\left(F_*\frac{\partial}{\partial t},\conjconn{u_t}_{e_i}u_{t*}e_i\right)\tau(u_t) + \overline{R}^{\substack{\scalebox{0.4}{\phantom{i}}\\N}}\left(F_*\frac{\partial}{\partial t},F_*{e_i}\right)\overline{\nabla}^{\substack{\scalebox{0.4}{\phantom{i}}\\u_t}}_{e_i}\tau(u_t).\\
        \end{split}
        \end{equation*}
    We also have
    \begin{equation}
        \label{seceq}
           \conjconn{F}_{\partialt}\operatorname{div}^g(T^M)\tau(u_t) = \operatorname{div}^g(T^M)\conjconn{F}_{\partialt}\tau(u_t),
    \end{equation}
    \begin{equation*}
        \conjconn{F}_{\partialt}K^N_{\tau(u_t)}\tau(u_t) = \left(\conjconn{N}_{F_*\partialt}K^N\right)\left(\tau(u_t),\tau(u_t)\right) + 2K^N_{\tau(u_t)}\conjconn{F}_{\partialt}\tau(u_t).
    \end{equation*}
    Consequently, it holds that
    \begin{equation}
        \label{threq}
        \begin{split}
            &\int_{\Omega} \left\langle F_*\partialt,\conjconn{F}_{\partialt}K^N_{\tau(u_t)}\tau(u_t)\right\rangle\,d\mu_g\\
            =&\int_{\Omega} \left\langle F_*\partialt,\left(\conjconn{N}_{F_*\partialt}K^N\right)\left(\tau(u_t),\tau(u_t)\right)\right\rangle\,d\mu_g + \int_{\Omega} \left\langle2K^N_{\tau(u_t)}F_*\partialt,\conjconn{F}_{\partialt}\tau(u_t) \right\rangle\,d\mu_g.
        \end{split}
    \end{equation}
    Lastly, using $(\ref{swap})$, we obtain
    \begin{equation}
        \label{foreq}
        \begin{split}
                -\conjconn{F}_{\partialt}L^N&\left(F_*{e_i},\tau(u_t)\right)F_*e_i\\
                =&\left(\conjconn{F}_{\partialt}\bar{L}^N\right)(\tau(u_t),F_*e_i)F_*e_i-L^N\left(F_*e_i,\conjconn{F}_{\partialt}\tau(u_t)\right)F_*e_i\\
                &-L^N\left(\conjconn{F}_{\partialt}F_*e_i,\tau(u_t)\right)F_*e_i-L^N(F_*e_i,\tau(u_t))\conjconn{F}_{\partialt}F_*e_i.
        \end{split}
    \end{equation}
    By combining the first item in $(\ref{fireq})$ and $(\ref{seceq})$, the following equation holds from Proposition $\ref{laplaceex}$.
    \begin{equation}
        \label{fifeq}
        \int_{\Omega} \left\langle F_*\partialt, \bar{\Delta}^{u_t}\conjconn{F}_{\partialt}\tau(u_t) + \operatorname{div}^g(T^M)\conjconn{F}_{\partialt}\tau(u_t)\right\rangle\,d\mu_g = \int_{\Omega} \left\langle\Delta^{u_t}F_*\partialt,\conjconn{F}_{\partialt}\tau(u_t)\right\rangle\,d\mu_g
    \end{equation}
    The third item of $(\ref{foreq})$ can be rewritten as
    \begin{equation}
        \label{sixeq}
        \begin{split}
            -&\int_{\Omega}\left\langle F_*\partialt, L^N\left(F_*e_i,\conjconn{F}_{\partialt}\tau(u_t)\right)F_*e_i\right\rangle\,d\mu_g\\
            &=-\int_{\Omega}\left\langle R^N\left(F_*e_i,F_*\partialt\right)F_*e_i,\conjconn{F}_{\partialt}\tau(u_t)\right\rangle\,d\mu_g.
        \end{split}
    \end{equation}
    By adding $(\ref{fifeq})$, $(\ref{sixeq})$, and the second item of $(\ref{threq})$ multiplied by $-1$, we obtain
    \begin{equation}
        \label{seveq}
        \int_{\Omega} \left\langle \Delta^{u_t}F_*\partialt +\sum_{i=1}^mR^N\left(F_*\partialt,F_*e_i\right)F_*e_i-2K^N_{\tau(u_t)}F_*\partialt,\conjconn{F}_{\partialt}\tau(u_t) \right\rangle\,d\mu_g
    \end{equation}
    We will now compute $\displaystyle \conjconn{F}_{\partialt}\tau(u_t)$.
    Since $\left[\partialt,e_i\right]=0$ holds, we have
    \begin{equation*}
        \begin{split}
            \nabla^F_{\partialt}\tau(u_t)&=\sum_{i=1}^m\nabla^F_{\partialt}\left(\nabla^F_{e_i}F_*{e_i}-F_*(\nabla^M_{e_i}e_i)\right)\\
            &=\sum_{i=1}^m\left(\nabla^F_*{e_i}\nabla^F_{\partialt}F_*(e_i)+R^N\left(F_*\partialt,F_*e_i\right)F_*e_i-\nabla^F_{\nabla^M_{e_i}e_i}F_*\partialt\right)\\
            &=\Delta^{u_t}F_*\partialt + \sum_{i=1}^mR^N\left(F_*\partialt,F_*e_i\right)F_*e_i.
        \end{split}
    \end{equation*}
    Since $\nabla^N-\conjconn{N}=2K^N$ holds, equation $(\ref{seveq})$ becomes
    \begin{equation*}
        \int_{\Omega} \bigg\|\Delta^{u_t}F_*\partialt +\sum_{i=1}^mR^N\left(F_*\partialt,F_*e_i\right)F_*e_i-2K^N_{\tau(u_t)}F_*\partialt\bigg\|^2_h\,d\mu_g.
    \end{equation*}
    We now derive $\mathcal{H}_{u}$ using the remaining items.
    By adding the second item in $(\ref{fireq})$ and the third item in $(\ref{foreq})$, by applying $(\ref{diffcurv})$ we obtain
    \begin{equation*}
        \begin{split}
                \overline{R}^{\substack{\scalebox{0.4}{\phantom{i}}\\N}}\left(\conjconn{F}_{e_i}F_*\partialt,F_*e_i\right)&\tau(u_t)-L^N\left(\conjconn{F}_{e_i}F_*\partialt,\tau(u_t)\right)F_*e_i\\
                =-&L^N(F_*e_i,\tau(u_t))\conjconn{F}_{e_i}F_*\partialt.
        \end{split}
    \end{equation*}
    Therefore, the remaining items are given by$:$
    \begin{equation*}
        \begin{split}
            &\sum_{i=1}^m\left(\overline{\nabla}^{\substack{\scalebox{0.4}{\phantom{i}}\\N}}_{F_*{e_i}}\overline{R}^{\substack{\scalebox{0.4}{\phantom{i}}\\N}}\right)\left(F_*\frac{\partial}{\partial t},F_*{e_i}\right)\tau(u_t)+\overline{R}^{\substack{\scalebox{0.4}{\phantom{i}}\\N}}\left(F_*\frac{\partial}{\partial t},\overline{\tau}(u_t)\right)\tau(u_t)\\
            &+2\sum_{i=1}^m\overline{R}^{\substack{\scalebox{0.4}{\phantom{i}}\\N}}\left(F_*\frac{\partial}{\partial t},F_*{e_i}\right)\overline{\nabla}^{\substack{\scalebox{0.4}{\phantom{i}}\\u_t}}_{e_i}\tau(u_t)-\sum_{i=1}^m\left(\conjconn{F}_{\partialt}\bar{L}^N\right)(\tau(u_t),F_*e_i)F_*e_i\\
            &-2\sum_{i=1}^mL^N(F_*e_i,\tau(u_t))\conjconn{F}_{e_i}F_*\partialt - \left(\conjconn{N}_{F_*\partialt}K^N\right)\left(\tau(u_t),\tau(u_t)\right).
        \end{split}
    \end{equation*}
Substituting $t=0$ yields $\mathcal{H}_u(V)$.\\
\end{proof}

We define stability of statistical biharmonic maps from the second variational formula.\\

\begin{definition}
    Let $u:(M,g,\nabla^M)\to(N,h,\nabla^N)$ be a statistical biharmonic map.
    The map $u$ is said to be \textit{weakly stable} if the right-hand side of \eqref{secvar} is non-negative for any relatively compact domain $\Omega\subset M$ and $V\in\Gamma_{\Omega}(u^{-1}TN)$.
    If the right-hand side of \eqref{secvar} is positive for any relatively compact domain $\Omega\subset M$ and variation with non-zero $V\in\Gamma_{\Omega}(u^{-1}TN)$, then $u$ is said to be \textit{stable}.\\
\end{definition}

\begin{proposition}
    A statistical biharmonic map $u:(M,g,\nabla^M)\to(N,h,\nabla^N)$ is called \textit{trivial} if $\tau(u)=0$.
    Every trivial statistical biharmonic map is weakly stable.\\
\end{proposition}

\section{Stability theorems of statistical biharmonic maps}
~~~We determine the stability of statistical biharmonic maps presented in subsection \ref{sbm}.
We also compute the second variational formula of the statistical bi-energy functional into statistical manifolds of constant sectional curvature, in particular into Hessian manifolds.

\subsection{Stability of improper affine spheres}
Let $\{f,\xi\}$ be a locally convex improper affine sphere, where $f:(M,g,\nabla)\to(\mathbb{R}^{m+1},g_0,D)$ is as in Example~\ref{impropaff}.
The statistical manifold $(M,g,\nabla)$ satisfies the equiaffine condition and $(\mathbb{R}^{m+1},g_0,D)$ is a flat Riemannian statistical manifold.
For a relatively compact domain $\Omega\subset M$, the second variational formula of the statistical bi-energy on $\Omega$ is
\begin{equation*}
    \left.\frac{d^2}{dt^2}\right|_{t=0} E_{2,\Omega}(f_t) = \int_{\Omega} \bigl\|\widehat{\Delta}^f V\bigr\|^2_{g_0}\,d\mu_g,
\end{equation*}
where $\{f_t\}_{t\in(-\epsilon,\epsilon)}$ is a variation that generates $V\in\Gamma_{\Omega}(f^{-1}T\mathbb{R}^{m+1})$ as the variational vector field.
The operator $\widehat{\Delta}^u$ is 
\begin{equation*}
    \widehat{\Delta}^fV = \sum_{i=1}^{m}\left(D^f_{e_i}D^f_{e_i}V-D^f_{\nabla^g_{e_i}e_i}V\right),
\end{equation*}
where $\{e_1,\ldots,e_m\}$ is an orthonormal frame of $g$.\\

\begin{theorem}
    Let $f:M\to\mathbb{R}^{m+1}$ be a locally convex improper affine sphere, and $\Omega\subset M$ a relatively compact domain.
    Equip $M$ with the induced statistical structure, and $(\mathbb{R}^{m+1},D)$ with the Euclidean metric $g_0$.
    Then, the mapping $f:(M,g,\nabla)\to(\mathbb{R}^{m+1},g_0,D)$ is a stable statistical biharmonic map.
\end{theorem}
\begin{proof}
    It is clear that $f$ is at least weakly stable.
    Suppose there is a relatively compact domain $\Omega\subset M$ and a variation $\{f_t\}$ of $f$ with the variational vector field $V\in\Gamma_{\Omega}(f^{-1}T\mathbb{R}^{m+1})$ that satisfies $\widehat{\Delta}^fV=0$.
    By Green's formula, we obtain
    \begin{equation*}
            0 = \int_{\Omega} \Delta_g\|V\|^2_{g_0}\,d\mu_g = \int_{\Omega} \bigl\|D^fV\bigr\|^2_{g_0}\,d\mu_g.
    \end{equation*}
    Hence, we have $D^fV=0$ on $\Omega$.
    Since the support of $V$ is compact in $\Omega$, the uniqueness of parallel transport by $D$ implies that $V=0$ on $\Omega$.\\
\end{proof}

\subsection{Reduction theorem for statistical biharmonic maps into conjugate symmetric statistical manifolds}
The Chen conjecture suggests that all biharmonic maps are harmonic maps in Riemannian geometry \cite{CBY}.
Although this statement is false in general, there are many studies identifying settings in which it holds \cite{CBY2, NHS, OT}.
Consider a mapping $u:(M,g,\nabla^M)\to(N,h,\nabla^N)$ where $(N,h,\nabla^N)$ is a statistical manifold of constant sectional curvature.
Then, we have $\nabla^hR^N=0$, and from Proposition \ref{Joper}, the inner product $\langle\mathcal{H}_u(V),V\rangle$ for $V\in\Gamma(u^{-1}TN)$ reduces to the following expression$\mathrm{:}$
\begin{equation}
    \begin{split}
            \langle\mathcal{H}_u(V),V\rangle &= \bigg\langle R^N(V,\overline{\tau}(u))\tau(u) + 2\sum_{i=1}^mR^N(V,u_*e_i)\widehat{\nabla}^u_{e_i}\tau(u)\\
            &- 2\sum_{i=1}^mR^N(u_*e_i,\tau(u))\widehat{\nabla}^u_{e_i}V - \left(\overline{\nabla}^{\substack{\scalebox{0.4}{\phantom{i}}\\N}}_{V}K^N\right)(\tau(u),\tau(u)),V\bigg\rangle.\\
        \end{split}
\end{equation}

\vspace{1em}

\begin{theorem}
    Let $u:(M,g,\nabla^M)\to(N,h,\nabla^N)$ be a statistical biharmonic map, where $(N,h,\nabla^N)$ is of constant sectional curvature $\lambda$.
    Assume that there exists a relatively compact domain $\Omega\subset M$ such that $\tau(u)\in\Gamma_{\Omega}(u^{-1}TN)$, $\Delta^u \tau(u) + \sum^m_{i=1}R^N(\tau(u),u_*{e_i})u_*{e_i}-2K^N_{\tau(u)}{\tau(u)}=0$, and suppose there exists a constant $C>0$ such that
    \begin{equation}
        \label{ineq1}
        4\lambda\,\|\tau(u)\|_h^2\,h(\tau(u),\widehat{\tau}(u)) + h\!\left(\left(\conjconn{N}_{\tau(u)}K^N\right)(\tau(u),\tau(u)),\,\tau(u) \right)
\ge C\,\|\tau(u)\|_h^4,
    \end{equation}
    where $\widehat{\tau}(u)$ is the tension field of $u:(M,g,\nabla^g)\to(N,\nabla^h)$.
    If $\langle u_*X,\tau(u)\rangle=0$ holds for any $X\in\Gamma(TM)$, then $u$ is weakly stable if and only if $\tau(u)=0$.\\
\end{theorem}

The inequality $(\ref{ineq1})$ holds for any map $u:(M,g,\nabla^M)\to(N,h,\nabla^N)$ if $(N,h,\conjconn{N})$ is a Hessian manifold of negative Hessian curvature.\\

\begin{proof}
    Let $\{u_t\}$ be a variation of $u$ that generates $V=\tau(u)$ as the variational vector field.
    Then, the second variational formula of the statistical bi-energy under the variation $\{u_t\}$ is
    \begin{equation*}
        \begin{split}
            &\left.\frac{d^2}{dt^2}\right|_{t=0} E_{2,\Omega}(u_t) \\
            &= \int_{\Omega} \bigg\langle \tau(u), 4\sum_{i=1}^mR^N(\tau(u),u_*e_i)\widehat{\nabla}^u_{e_i}\tau(u) - \left(\overline{\nabla}^{\substack{\scalebox{0.4}{\phantom{i}}\\N}}_{\tau(u)}K^N\right)(\tau(u),\tau(u))\bigg\rangle\,d\mu_g.
        \end{split}
    \end{equation*}
    We have
    \begin{equation*}
        \begin{split}
        \sum_{i=1}^m\bigg\langle\tau(u),R^N(\tau(u),u_*e_i)\widehat{\nabla}^u_{e_i}\tau(u)\bigg\rangle &= \lambda\sum_{i=1}^mh(\widehat{\nabla}^u_{e_i}\tau(u),u_*e_i)\|\tau(u)\|_h^2.\\
        &=-\lambda\sum_{i=1}^mh(\tau(u),\widehat{\nabla}^u_{e_i}u_*e_i)\|\tau(u)\|_h^2\\
        &=-\lambda h(\tau(u),\widehat{\tau}(u))\|\tau(u)\|_h^2,
        \end{split}
    \end{equation*}
    thus, 
    \begin{equation*}
            0\leq\left.\frac{d^2}{dt^2}\right|_{t=0} E_{2,\Omega}(u_t)\leq-C\int_{\Omega} \|\tau(u)\|_h^4\,d\mu_g\leq0.
    \end{equation*}
    Hence, we obtain $\tau(u)=0$ on $M$.
\end{proof}

\subsection{Stability of statistical biharmonic maps into Hessian manifolds}
Since $R^N=\overline{R}^{\substack{\scalebox{0.4}{\phantom{i}}\\N}}=0$ holds for any Hessian manifold $(N,h,\nabla^N)$, the following proposition holds immediately.\\

\begin{proposition}
    Let $u:(M,g,\nabla^M)\to(N,h,\nabla^N)$ be a statistical biharmonic map, where $(N,h,\nabla^N)$ is a Hessian manifold.
    Then, the operator $\mathcal{H}_u$ of $u$ reduces to the following$\mathrm{:}$
    \begin{equation*}
        \mathcal{H}_u(V) = -\left(\conjconn{N}_{V}K^N\right)\left(\tau(u),\tau(u)\right),\quad V\in\Gamma(u^{-1}TN).
    \end{equation*}
\end{proposition}

\vspace{1em}

We now obtain two corollaries related to Hessian manifolds of constant Hessian curvature.\\

\begin{cor}
    \label{cor1}
    Suppose $(N,h,\conjconn{N})$ is a Hessian manifold with Hessian curvature $\overline{H}$.
    Then, for a map $u:(M,g,\nabla^M)\to(N,h,\nabla^N)$, the operator $\mathcal{H}_u$ satisfies\\
    \begin{equation*}
        \mathcal{H}_u(V) = \overline{H}(\tau(u),\tau(u);V),\quad V\in\Gamma(u^{-1}TN).
    \end{equation*}
    In particular, if $(N,h,\conjconn{N})$ is CHC $c$, then $\langle\mathcal{H}_u(V),V\rangle = -c\cdot h(\tau(u),V)^2,$ $V\in\Gamma(u^{-1}TN)$ holds.
\end{cor}

\vspace{1em}

\begin{cor}
    \label{cor2}
    If $(N,h,\nabla^N)$ is a Hessian manifold of CHC $c$, then the operator $\mathcal{H}_u$ satisfies\\
    \begin{equation*}
        \langle\mathcal{H}_u(V),V\rangle = \frac{c}{2}(\|V\|_h^2\|\tau(u)\|_h^2+h(V,\tau(u))^2)-2\|K^N_V \tau(u)\|_h^2,
    \end{equation*}
    for $V\in\Gamma(u^{-1}TN)$.
\end{cor}
\begin{proof}
    For any Hessian manifold $(N,h,\nabla^N)$, we have
    \begin{equation*}
        \begin{split}
            \left(\conjconn{N}_XK^N\right)(Y,Z) &= \left(\nabla^N_XK^N\right)(Y,Z) - 2[K^N_X,K^N_Y]Z + 2K^N_ZK^N_XY\\
            &=\left(\nabla^N_YK^N\right)(X,Z) + 2K^N_ZK^N_XY,
        \end{split}
    \end{equation*}
    for any $X,Y,Z\in\Gamma(TN)$.
    Hence, if the Hessian manifold $(N,h,\nabla^N)$ is CHC $c$, by equation $(\ref{chcc})$ we obtain the desired equation.\\
\end{proof}


\begin{example}
    Consider the curve $\gamma(t)=(\lambda_1t^2,\ldots,\lambda_mt^2),\,t\in(0,\infty)$ in Example~\ref{curvhess0}.
    The map $\gamma:((0,\infty),g_0,\nabla^{g_0})\to((\mathbb{R}^{+})^m,g_0,\nabla)$ is a stable statistical biharmonic map.
    Indeed, since $((\mathbb{R}^{+})^m,g_0,\overline{\nabla})$ is a Hessian manifold of CHC $0$, it follows from Corollary~\ref{cor1} that $\gamma$ is at least weakly stable.
    Let $b>a>0$ and set $I=(a,b)$.
    Take any $V\in\Gamma_{I}(\gamma^{-1}T((\mathbb{R}^{+})^m))$ and let $\{\gamma_s\}$ be a variation of $\gamma$ that generates $V$ as the variational vector field.
    The vector field $V$ can be expressed using $\phi^1,\ldots,\phi^m\in C^{\infty}(\mathbb{R})$ with as follows $\mathrm{:}$
    \begin{equation*}
        V=\sum_{i=1}^m\phi^i\frac{\partial}{\partial x^i}.
    \end{equation*}
    Each $\phi^i$ has compact support in $I$.
    With this expression, the second variational formula of $E_2$ under the variation $\gamma_s$ is
    \begin{equation}
        \label{secvar-h}
        \begin{split}
        \left.\frac{d^2}{ds^2}\right|_{s=0} E_{2,I}(\gamma_s) =& \int_{a}^{b} \biggl\|\sum_{i=1}^m\nabla_{\dot{\gamma}}\nabla_{\dot{\gamma}} \phi^i\frac{\partial}{\partial x^i} -2\sum_{i=1}^mK_{\nabla_{\dot{\gamma}}\dot{\gamma}}\phi^i\frac{\partial}{\partial x^i}\biggr\|^2_{g_0}\,dt\\
        =& \int_{a}^{b} \biggl\|\sum_{i=1}^m\left(\ddot{(\phi^i)} + \frac{4\dot{(\phi^i)}}{t} - \frac{12\phi^i}{t^2}\right)\frac{\partial}{\partial x^i}\biggr\|^2_{g_0}\,dt\\
        =& \sum_{i=1}^m\int_{a}^{b} \left(\ddot{(\phi^i)} + \frac{4\dot{(\phi^i)}}{t} - \frac{12\phi^i}{t^2}\right)^2\,dt.
        \end{split}
    \end{equation}
    By solving the linear second-order differential equation, in order for the right-hand side of equation \eqref{secvar-h} to vanish, each $\phi^i$ must be
    \begin{equation}
        \phi^i(t) = C_1^it^{\mu_+} + C_2^it^{\mu_-},
    \end{equation}
    where $C_1^i,C^i_2$ are constants and $\mu_{\pm}=\frac{1}{2}(-3\pm\sqrt{57})$.
    Since each $\phi^i$ has compact support in $I$, we obtain $C^i_1=C^i_2=0$, thus $V=0$.
    Therefore, the curve $\gamma$ is stable.\\
\end{example}

\begin{example}
    Consider the curve $\gamma(t)=(0,e^{\lambda t}),\,t\in\mathbb{R}$ on the statistical manifold of normal distributions $(\mathbb{H}^2,g^F,\nabla^{(m)})$ in Example~\ref{projcurv}.
    Define a vector field along $\gamma$ by
    \begin{equation*}
        X = e^{\lambda t}\frac{\partial}{\partial x},\quad t\in\mathbb{R}.
    \end{equation*}
    Then, the pair $\{\dot{\gamma},X\}$ is an orthogonal frame along $\gamma$ parallel with respect to $\nabla^g$, and we have $\|\dot{\gamma}\|^2_{g^F}=2\lambda^2$ and $\|X\|^2_{g^F}=1$ by definition.
    Let $I=(a,b)\in\mathbb{R}$ be bounded.
    The variational vector field $V\in\Gamma_{I}(\gamma^{-1}T\mathbb{H}^2)$ generated by a smooth variation $\gamma_s$ of $\gamma$ can be expressed by $V=\phi\dot{\gamma}+ \psi X$ using $\phi,\psi\in C^{\infty}(\mathbb{R})$ which has compact support in $I$.
    Using this expression, the second variational formula of $E_2$ under the variation $\gamma_s$ is
    \begin{equation}
        \label{secvar-m}
        \begin{split}
        \left.\frac{d^2}{ds^2}\right|_{s=0} E_{2,I}(\gamma_s) =& \int_{a}^{b} \biggl\|\nabla^{(m)}_{\dot{\gamma}}\nabla^{(m)}_{\dot{\gamma}}V -2K^F_{\nabla^{(m)}_{\dot{\gamma}}\dot{\gamma}}V\biggr\|^2_{g^F}\,dt\\
        &-\int_{a}^{b}g^F\left((\nabla^{(e)}_{V}K^F)(\nabla^{(m)}_{\dot{\gamma}}\dot{\gamma},\nabla^{(m)}_{\dot{\gamma}}\dot{\gamma}),V\right)\,dt\\
        =& \int_{a}^{b} \biggl\|\ddot{\phi}\dot{\gamma} + 4\lambda\dot{\phi}\dot{\gamma} - 4\lambda^2\phi\dot{\gamma}+\ddot{\psi}X + 2\lambda\dot{\psi}X - 3\lambda^2\psi X\biggr\|^2_{g^F}\,dt\\
        -&\int_{a}^{b} \left(4\lambda^2\|V\|_{g^F}^2\|\dot{\gamma}\|_{g^F}^2+4\lambda^2g^F(V,\dot{\gamma})^2-8\lambda^2\|K^F_{\dot{\gamma}} V\|_{g^F}^2\right)\,dt.
        \end{split}
    \end{equation}
    The last row is obtained by Corollary~\ref{cor2} and that $\nabla^{(m)}_{\dot{\gamma}}\dot{\gamma} = K^F_{\dot{\gamma}}\dot{\gamma} =2\lambda\dot{\gamma}$, $\nabla^{(m)}_{\dot{\gamma}}X=K^F_{\dot{\gamma}}X=\lambda X$.
    Here, we have
    \begin{equation*}
            \|V\|_{g^F}^2\|\dot{\gamma}\|_{g^F}^2+g^F(V,\dot{\gamma})^2-2\|K^F_{\dot{\gamma}} V\|_{g^F}^2=-8\lambda^4\phi^2.
    \end{equation*}
    Since $\phi,\psi$ has compact support in $I$, \eqref{secvar-m} is equal to
    \begin{equation*}
        \begin{split}
        \int_{a}^{b}& \bigg(2\lambda^2\bigg(\ddot{\phi} + 4\lambda\dot{\phi} - 4\lambda^2\phi\bigg)^2-32\lambda^6\phi^2\bigg)\,dt+\int_{a}^{b} \bigg(\ddot{\psi} + 2\lambda\dot{\psi} - 3\lambda^2\psi\bigg)^2\,dt\\
        =&2\lambda^2\int_{a}^{b} (\ddot{\phi})^2\,dt + 48\lambda^4\int_{a}^{b}(\dot{\phi})^2\,dt + \int_{a}^{b} \bigg(\ddot{\psi} + 2\lambda\dot{\psi} - 3\lambda^2\psi\bigg)^2\,dt.
        \end{split}
    \end{equation*}
    Hence, if \eqref{secvar-m} is $0$, then $\phi=\psi=0$ must hold, thus $V=0$.
    Therefore, the statistical biharmonic map $\gamma:(\mathbb{R},g_0,\nabla^{g_0})\to(\mathbb{H}^2,g^F,\nabla^{(m)})$ is stable.

    The map $\gamma:(\mathbb{R},g_0,\nabla^{g_0})\to(\mathbb{H}^2,g^F,\nabla^{(e)})$ was also a statistical biharmonic map.
    Let $I=(a,b)\subset\mathbb{R}$ be bounded and $\gamma_s$ be a smooth variation of $\gamma$ that generates the variational vector field $V=\phi\dot{\gamma}+ \psi X$ where $\phi,\psi\in C^{\infty}(\mathbb{R})$ have compact domain in $I$.
    Since $(\mathbb{H}^2,g^F,\nabla^{(m)})$ is CHC $2$, Corollary~\ref{cor1} implies that the second variational formula under $\gamma_s$ is
    \begin{equation*}
        \begin{split}
            \left.\frac{d^2}{ds^2}\right|_{s=0} E_{2,I}(\gamma_s) =& \int_{a}^{b} \biggl\|\nabla^{(e)}_{\dot{\gamma}}\nabla^{(e)}_{\dot{\gamma}}V + 2K^F_{\nabla^{(e)}_{\dot{\gamma}}\dot{\gamma}}V\biggr\|^2_{g^F}\,dt\\
            &+\int_{a}^{b}g^F\left((\nabla^{(m)}_{V}K^F)(\nabla^{(e)}_{\dot{\gamma}}\dot{\gamma},\nabla^{(e)}_{\dot{\gamma}}\dot{\gamma}),V\right)\,dt\\
            =&2\lambda^2\int_{a}^{b} \bigg(\ddot{\phi}-4\lambda\dot{\phi}\bigg)^2\,dt + 16\lambda^4\int_{a}^{b}(\dot{\phi})^2\,dt \\
            =&2\lambda^2\int_{a}^{b} (\ddot{\phi})^2\,dt + 48\lambda^4\int_{a}^{b}(\dot{\phi})^2\,dt \\
            &+ \int_{a}^{b} \bigg(\ddot{\psi} - 2\lambda\dot{\psi} - 3\lambda^2\psi\bigg)^2\,dt.
        \end{split}
    \end{equation*}
    Thus, $V=0$ must hold in order for this integral to be non-positive, hence we conclude the map $\gamma:(\mathbb{R},g_0,\nabla^{g_0})\to(\mathbb{H}^2,g^F,\nabla^{(e)})$ is a stable statistical biharmonic map.
\end{example}

\vspace{1em}
\noindent\scalebox{1.2}{Acknowledgments}\\
\par
\vspace{-0.5em}
\noindent The author would like to express their gratitude to Professor Hajime Urakawa for encouragement and valuable discussions.\\
\par
\vspace{-0.5em}
\noindent Funding.\,
This work was supported by the Japan Science and Technology Agency (JST), the establishment of university fellowships towards the creation of science technology innovation, Grant Number JPMJSP2119.


\end{document}